\documentclass[final,leqno]{siamltex1213}

\usepackage{amsfonts,amssymb,amsmath,mathtools}
\usepackage{subcaption}
\usepackage{framed}
\usepackage{enumitem}
\usepackage{paralist}
\usepackage{pgfplots}

\newtheorem{rem}[theorem]{Remark}

\usepackage{soul}

\newcommand{\Xitrain}{{\Xi_\mathrm{train}}}
\newcommand{\Xitrainanchor}{{\Xi_\mathrm{train}^\mathrm{anchor}}}

\usepackage{algorithm}
\usepackage{algorithmic}

\DeclareMathAlphabet{\mathpzc}{OT1}{pzc}{m}{it}

\newcommand{\X}{{\mathbb{X}}}
\newcommand{\Y}{{\mathbb{Y}}}
\newcommand{\R}{{\mathbb{R}}}

\newcommand{\N}{{\mathbb{N}}}

\newcommand{\E}{\mbox{I\negthinspace E}}

\newcommand{\cD}{{\mathcal{D}}}
\newcommand{\cG}{{\mathcal{G}}}
\newcommand{\cH}{{\mathcal{H}}}

\newcommand{\cN}{{\mathcal{N}}}

\newcommand{\cT}{{\mathcal{T}}}

\newcommand{\cX}{{\mathcal{X}}}

\newcommand{\bx}{\boldsymbol{x}}

\graphicspath{{Images/}}

\title{A Reduced Basis Method for the Hamilton-Jacobi-Bellman Equation with Application to the European Union Emission Trading Scheme\footnotemark[1]}

\author{Sebastian Steck\footnotemark[3]
	\and
	Karsten Urban\footnotemark[4]}

\begin{document}
\maketitle

\renewcommand{\thefootnote}{\fnsymbol{footnote}}
\footnotetext[3]{University of Ulm, Institute for Numerical Mathematics, Helm\-holtz\-strasse 20, D-89069 Ulm, Germany, {\tt{sebastian.steck@uni-ulm.de}}}
\footnotetext[4]{University of Ulm, Institute for Numerical Mathematics, Helm\-holtz\-strasse 20, D-89069 Ulm, Germany, {\tt{karsten.urban@uni-ulm.de}}}

\footnotetext[1]{This research was supported by Deutsche Forschungsgemeinschaft (DFG) within SPP 1324 ``Mathematical methods for extracting quantifiable information from complex systems''.}
\renewcommand{\thefootnote}{\arabic{footnote}}

\begin{abstract}
This paper draws on two sources of motivation: (1) The European Union Emission Trading Scheme (EU-ETS) aims at limiting the overall emissions of greenhouse gases. The optimal abatement strategy of companies for the use of emission permits can be described as the viscosity solution of a Hamilton-Jacobi-Bellman (HJB) equation. It is a question of general interest, how regulatory constraints can be set within the EU-ETS in order to reach certain political goals such as a good balance of emission reduction and economical growth. Such regulatory constraints can be modeled as parameters within the HJB equation. 

(2) The EU-ETS is just one example where one is interested in solving a parameterized HJB equation often for different values of the parameters (e.g.\ to optimize their values with respect to a given target functional). The Reduced Basis Method (RBM) is by now a well-established numerical method to efficiently solve parameterized partial differential equations. However, to the best of our knowledge, an RBM for the HJB equation is not known so far and of (mathematical) interest by its own, since the HJB equation is of hyperbolic type which is in general a nontrivial task for model reduction.

We analyze and realize a RBM for the HJB equation. In particular, we construct an online-efficient error estimator for this nonlinear problem using the Brezzi-Rapaz-Raviart (RBB) theory as well as numerical algorithms for the involved parameter-dependent constants. Numerical experiments are presented.
\end{abstract}



\date{\today}



\section{Introduction}\label{Sec:1}

A driving source of motivation for the investigations reported in this paper is the European Union Emission Trading Scheme (EU-ETS)  that has been invented in order to limit the emission of greenhouse gases. Within this EU-ETS, a limited amount of emission permits is issued and each pollutant needs to cover its emissions with sufficient permits. Both from the environmental and ecological as well as from the economical point of view it is important to control the EU-ETS in such a way that certain desired political effects are reached. As a simple example, the number of permits should limit the emission of global warming gases without leading to a severe economical crises. Thus, we are interested in investigating the effect of several different regulatory constraints (i.e., a parameter study from a mathematical point of view) as well as trying to find regulatory strategies in order to reach certain goals (i.e., realtime optimal control).
 
From a mathematical point of view, this means that the same model has to be solved for a variety of parameters, here different regulatory constraints. We describe these constraints in terms of parameters $\mu\in\cD$, where $\cD\subset\R^P$ is the set of all possible parameter values. Moreover, a full mathematical model of the EU-ETS is in general complex so that the numerical simulation is costly and parameter studies as well as realtime optimal control is not feasible. Hence, we suggest to use the Reduced Basis Method (RBM), a model reduction technique that uses a possible costly offline phase in order to computationally construct a reduced system, which is then used in the multi-query (parameter study) or realtime (optimal control) context in order to produce numerical approximations for various parameter values highly efficient and with mathematical certification in terms of a posteriori error control.
 
It turns out that an optimal abatement strategy concerning the use of emission permits can be described by the Hamilton-Jacobi-Bellman (HJB) equation, parameterized by the regulatory constraints. This brings us to the second source of motivation for this paper, which is of mathematical nature and independent of the specific application of the EU-ETS. In fact, to the best of our knowledge, an RBM for the HJB equation is not known so far and of interest by its own since the HJB equation is of hyperbolic type which is in general a nontrivial task for model reduction.

It is the aim of this paper to construct, analyze and realize a RBM for the HJB equation with the specific application of the EU-ETS. The remainder of this paper is organized as follows. In Section \ref{Sec:2}, we collect the required preliminaries on the mathematical model for the EU-ETS yielding the HJB equation. Section \ref{Sec:3} is devoted to the discretization including an error analysis. Since we are facing a nonlinear problem (in the specific case of emission trading a quadratic problem, see also Remark \ref{Rem:quad} below), we use the Brezzi-Rappaz-Raviart (BRR) theory in order to ensure well-posedness and error control.  The RBM for the HJB equation is introduced in Section \ref{Sec:4} and finally, in Section \ref{Sec:5} we present results of some numerical experiments.

\section{Preliminaries}\label{Sec:2}
In this section, we collect some preliminaries.

\subsection{Mathematical Model for the European Union Emission Trading System (EU-ETS)}\label{Sec:2.1}
We start by introducing a mathematical model for the European Union Emission Trading System (EU-ETS) and show that the market equilibrium can be described in terms of a Hamilton-Jacobi-Bellman (HJB) equation. First, the emission trading system is organized in trading periods, but for simplicity we may reduce ourselves to one period only. It can be shown under reasonable assumptions, that the market equilibrium for one single trading period $[0,T]$ can be characterized by the fact that the sum of the costs of all market participants is minimal, \cite{CarmonaFehrHinzPorchet}. 

Let $Y_\tau$ be a stochastic process, $\tau\in[0,T]$, which describes the amount of uncovered emissions, sometimes also called \emph{state}. For a set of $d$ companies whose greenhouse gas emissions are considered, the values of $Y_\tau$ are thus taken in $\R^d$. The $\R^d$-valued stochastic control $\pi_\tau$ describes the additional abatement of emissions compared to the so-called \emph{business as usual} strategy. Hence, an \emph{optimal abatement strategy} $\pi=(\pi_\tau)_{\tau\in [0,T]}$ should minimize the \emph{expected abatement costs}, i.e., the cost functional
\begin{equation}\label{eq:J1}
	J(\pi):= 
	\E \left[\int_0^T  f^{\pi}\left(\tau, Y_\tau\right) d\tau + h(Y_T) \right],
\end{equation}
where $f^{\pi}$ denotes the running abatement cost using strategy $\pi$ and the function $h$ models the penalty to be paid at the end of the trading period.

\begin{rem}\label{Rem:quad}
	For later reference, we note that in the specific case of the EU-ETS, the dependency of $f^{\pi}$ with respect to $\pi$ is quadratic.
\end{rem}

A standard stochastic model for the amount of uncovered emissions $Y_\tau$ reads
\begin{equation}\label{eq:SDE1}
	dY_\tau = b^{\pi}(\tau, Y_\tau) d\tau + \sigma^{\pi}(\tau, Y_\tau) dW_\tau,\, \tau\in (0,T], 
	\qquad 
	Y_0 = y_0, 
\end{equation}
where $W_\tau$ is a $d$-dimensional Wiener process and $b^{\pi}$, $\sigma^{\pi}$ are coefficients for drift and volatility, respectively, such that $b^\pi$ is linear w.r.t.\ the control $\pi$ as well as $\sigma^{\pi}(\sigma^{\pi})^T$ is linear in $\pi$.

The transition to a partial differential equation (PDE) in terms of a HJB equation is then done by introducing new variables $(t,x)$ for time and state and to modify \eqref{eq:SDE1} to
\begin{equation}\label{eq:SDE2}
	dY_\tau = b^{\pi}(\tau, Y_\tau) d\tau + \sigma^{\pi}(\tau, Y_\tau) dW_\tau,\, \tau\in (t,T], 
	\qquad 
	Y_t = x, 
\end{equation}
i.e., the initial time $t$ and the initial state $x$ at $\tau=t$ are the new variables. Moreover, the stochastic control $\pi$ is replaced by a determinstic (but state-dependend) function $\gamma: (t,x)\mapsto \R^d$ and accordingly, the cost functional in \eqref{eq:J1} reads
\begin{equation}\label{eq:J2}
	J(t,x; \gamma):= 
	\E \left[\int_t^T  f^{\gamma}\left(\tau, Y_\tau\right) d\tau + h(Y_T) \right],
\end{equation}
where the dependency of $x$ is implicit via $Y_\tau=Y_\tau(t,x)$ by \eqref{eq:SDE2}. 
In order to derive a strategy that yields minimal cost, one needs to solve a stochastic optimization problem whose solution is the \emph{value function}, i.e., for $x\in\R^d$
\begin{equation}\label{Eq:ValueFunction}
	u(t,x) = \inf_{\gamma\in\Gamma} J(t,x; \gamma) \quad\forall t\in[0, T),
	\qquad
	u(T,x) = h(x),
\end{equation}
where $\Gamma\subset L_\infty([0,T]\times\R^d; \R^d)$ is a suitable set of admissible controls. 
It is well-known that the value function is a viscosity solution of the Hamilton-Jacobi-Bellman (HJB) equation (see e.g.\ \cite[Theorem IV.5.2]{YongZhou})
\begin{equation}
	\label{Eq:HJB1}
	\partial_t u(t,x) + \sup_{\gamma\in\Gamma} \Big\{ \frac12 \text{tr}(\sigma^{\gamma}(\sigma^{\gamma})^T\, \nabla^2u(t,x)) 
		+ b^{\gamma}\cdot\nabla u(t,x) 
		- f^{\gamma}(t,x)\Big\}=0,
\end{equation}
for all $(t,x)\in [0,T]\times\Omega$, where $\nabla^2u$ denotes the Hessian of $u$. 
We consider a bounded domain $\Omega\subset\R^d$ for the state. The reason is twofold: (1) The limit of the price for emissions exceeding the available permits is the penalty set by the regulating authorities; (2) If, on the other hand, there is a vast excess of permits, no emissions will be saved and the value of the permits is only determined by their terminal value at $T$. Of course, using a bounded domain $\Omega\subset\R^d$ implies the need to set appropriate boundary conditions. Since the price for the permits is determined by the derivative of the value function $u$ w.r.t.\ the need for permits $x$, a corresponding Neumann boundary conditions are appropriate, see \eqref{Eq:HJB:2} below.

Last, but not least, we model the appearance of parameters $\mu\in\cD\subset\R^P$ as already introduced in Section \ref{Sec:1}. This means that basically all quantities may be $\mu$-dependent, e.g.\ $f^\gamma(\mu)$, $b^\gamma(\mu)$, $\sigma^\gamma(\mu)$, $J(\mu; t,x; \gamma)$ and the value function $u(\mu)=u(\mu; t,x)$.

\subsection{Hamilton-Jacobi-Bellman (HJB) Equation}\label{Sec:2.2}

Let $T>0$ be some final time, $I:=[0,T]$ the time interval, $\Omega\subset\R^d$ be a bounded domain and denote by $\Omega_T:=I\times\Omega$ the time-space domain. The parameters are denoted by $\mu\in\cD$, where $\cD\subset\R^P$ is the parameter space.

Next, let $\Gamma\subset L_\infty(\Omega_T; \R^d)$ be a separable, complete metric space, the space of admissible controls, such that the mapping 
$$
	\Gamma\to C(\bar\Omega)\times C(\bar\Omega, \R^d)\times C(\bar\Omega)\times C(\bar\Omega), \quad
	\Gamma\ni \gamma\mapsto (a^\gamma(\mu), b^\gamma(\mu), c^\gamma(\mu), f^\gamma(\mu))
	$$ 
is continuous for any parameter $\mu\in\cD$. We define the linear parameterized PDE (PPDE) operator in space by 
$$
	u\mapsto A^\gamma(\mu; u) := -a^\gamma(\mu)\, \Delta u + b^\gamma(\mu)\cdot \nabla u + c^\gamma(\mu)u, 
	\quad\mu\in\cD.
$$
In the application of the EU-ETS described in \S\ref{Sec:2.1} above, this operator takes the form $A^\gamma(\mu; u):=\frac12 \text{tr}(\sigma^\gamma(\mu)(\sigma^\gamma(\mu))^T\, \nabla^2u) + b^\gamma(\mu)\cdot\nabla u$. Note, that $A^\gamma(\mu; u)$ is linear in $u$, whereas the dependency on the control $\gamma$ is in general not linear. Finally, we define the nonlinear parameterized Hamilton-type operator by 
\begin{equation}
	\label{Eq:DefH}
	\cH(\mu; u):= \sup_{\gamma\in\Gamma} \{A^\gamma(\mu; u) - f^\gamma(\mu)\},
\end{equation}
as well as the linear space-time differential operator $L^\gamma(\mu; u):= \partial_t u + A^\gamma(\mu; u)$, which is linear in $u$, but not self-adjoint. The range of this operator should also contain boundary and terminal conditions included in the following set of Hamilton-Jacobi-Bellman (HJB) equations (compare \eqref{Eq:HJB1})
\begin{subequations}
	\label{Eq:HJB}
	\begin{align}
		\partial_t u + \cH(\mu; u) &= 0, && \hspace*{-20mm} \text{in } \Omega_T, \label{Eq:HJB:1} \\
		\frac\partial{\partial n} u &= \psi, && \hspace*{-20mm}\text{on } \partial\Omega_T=(0,T)\times\partial\Omega, \label{Eq:HJB:2} \\
		u(T) &= u_T, && \hspace*{-20mm}\text{on } \bar\Omega, \label{Eq:HJB:3}
	\end{align}
\end{subequations}
where $\psi$ is a suitable function modeling the Neumann truncation boundary conditions and $u_T\in H^1(\Omega)$ denotes the terminal condition (which, in particular needs to be compatible with $\psi$ on $\{T\}\times\partial\Omega$). For the correct interpretation of the subsequent discussion, it is worthwhile to detail \eqref{Eq:HJB:1} in combination with \eqref{Eq:DefH}:
\begin{equation}
	\label{Eq:HJB:1a}
	\partial_t u(t,x) + \sup_{\gamma(t,x)\in\R^d} \{ A^{\gamma(t,x)}(\mu; u(t,x)) - f^{\gamma(t,x)}(\mu) \}=0, \,\,\, \forall (t,x)\in\Omega_T.
\end{equation}

This shows that the control space is $\Gamma:= L_\infty(\Omega_T; \R^d)$, or an appropriate subspace of the latter. 
We can also write \eqref{Eq:HJB} in the following form
\begin{equation}
	\label{Eq:HJBa}
	\sup_{\gamma\in \Gamma} \{ L^\gamma(\mu; u) - g^\gamma(\mu)\}=0\quad \text{on } \Omega_T.
\end{equation}

Given a value function, i.e., a solution $u^\ast\in U$ (where $U$ is an appropriate solution space, e.g.\ $H^1(I; H^{-1}(\Omega))\cap L_2(I; H^1_0(\Omega))$) of \eqref{Eq:HJB}, the corresponding optimal control $\gamma^\ast\in\Gamma$ is given by
$$
	\gamma^\ast = \arg \sup_{\gamma\in\Gamma} \{ L^\gamma(\mu; u^\ast) - g^\gamma(\mu)\}.
$$
The pair of optimal value function and optimal control is also written as $x^\ast=(\gamma^\ast, u^\ast)\in\X:=\Gamma\times U$.

\subsubsection*{Well-posedness}
It is well-known that well-posedness of the HJB-equation \eqref{Eq:HJB} is ensured under reasonable assumptions. In fact, if $A^\gamma$, $f^\gamma$ and $u_T$ are uniformly continuous and uniformly Lipschitz continuous with respect to the state $x\in\R^d$ as well as bounded at the state $x=0$, it was proven e.g.\ in \cite[Theorem IV.6.1]{YongZhou} that the HJB equation \eqref{Eq:HJB} admits at most one viscosity solution such that there exists constant $K\in\R$ with
\begin{subequations} \label{Equ:HJBlsg-schranken}
\begin{eqnarray}
	| u(t,x) | &\leqslant& K(1+|x|) \\
	| u(t,x) - u(\hat t, \hat x) | &\leqslant& K \big( |x-\hat x| + (1+\min\{|x|,|\hat x|\}) {|t-\hat t|}^{1/2} \big)
\end{eqnarray}
\end{subequations}
for all $(t,x),(\hat t,\hat x)\in\Omega_T$.

Furthermore, the value function $u$ in \eqref{Eq:ValueFunction} is a viscosity solution of the HJB equation \eqref{Eq:HJB}, which satisfies \eqref{Equ:HJBlsg-schranken} \cite[Theorem IV.5.2 \& Proposition IV.3.1]{YongZhou}. Thus, the value function  \eqref{Eq:ValueFunction} is the unique solution of the HJB equation.

\begin{rem}
	It should be noted that we do not have a linear-quadratic problem even though $f^\gamma$ is quadratic in $\gamma$, the HJB remains nonlinear and we cannot expect the availability of a solution formula or even a smooth solution, \cite{YongZhou}.
\end{rem}

\section{Discretization}\label{Sec:3}
We now describe the essentials of a numerical discretization. We start by a possibly high-dimensional model that is assumed to reflect the true model sufficiently well, thus called `truth' discretization. This will later be the basis for model reduction.

\newcommand{\myGamma}{{\R^{d}}}
\newcommand{\myGammaN}{{\R^{d\cN}}}

\subsection{`Truth' discretization}
Recall from \eqref{Eq:HJB:1a} that the supremum is taken pointwise. This also motivates that most discretizations are pointwise. In order to describe such schemes, let $z_i:=(t_i,x_i) \in\Omega_T$, $i=1,\ldots,\cN$, be a set of points in the time-space domain, where $\cN\gg 1$ is assumed to be `large', in particular large enough to represent the main characteristics of the continuous problem \eqref{Eq:HJBa} as well as `too large' for multi-query or realtime simulations. Hence, we are looking for a discrete approximation $u^{\ast,\cN}(\mu) = (u_i^{\ast,\cN}(\mu))_{i=1,\ldots, \cN}\in\R^\cN$ of the value function $u^\ast\in U$ at the points $z_i=(t_i,x_i)$. This means that a pointwise discretization of \eqref{Eq:HJBa} amounts solving an optimization problem for \emph{each} point, i.e., we need to determine a component of an approximation to $\gamma^*\in\Gamma$ for each $i$. We denote such an approximation of $\gamma^\ast(t_i,x_i)$ by $\gamma_i^{\ast, \cN}(\mu)$. If we abbreviate the pointwise evaluation of the operator and right-hand side, respectively, as
\begin{align*}
	& L^{\cN,\gamma_i}(\mu;\cdot): \R^\cN\to\R,
	&& L^{\cN,\gamma_i}(\mu;\cdot) := [L^{\gamma(t_i,x_i)}(\mu;\cdot)](t_i,x_i), \\
	& f^{\cN,\gamma_i}(\mu) \in\R,
	&& f^{\cN,\gamma_i}(\mu) := [f^{\gamma(t_i,x_i)}(\mu)](t_i,x_i),
\end{align*}
we obtain the discretized optimization problem of dimension $\cN$:
\begin{equation}
	\label{Eq:HJB:cN}
	\text{Find } u^\cN(\mu)\in\R^\cN: \quad
	\max_{\gamma_i \in \myGamma} \{ L^{\cN,\gamma_i}(\mu; u^\cN(\mu)) - f^{\cN,\gamma_i}(\mu) \}=0.
	\quad \forall 1\le i\le\cN.
\end{equation}
The corresponding solution is denoted by $u^{\ast,\cN}(\mu) = (u_i^{\ast,\cN}(\mu))_{i=1,\ldots, \cN}\in\R^\cN$, which is a discrete approximation  of the value function $u^\ast\in U$. 
For the optimal control, we set (with the solution $u^{\ast, \cN}(\mu)$ of \eqref{Eq:HJB:cN}) $\gamma^{\ast, \cN}(\mu) = (\gamma_i^{\ast, \cN}(\mu))_{i=1,\ldots, \cN}\in\myGammaN$ defined for each $1\le i\le\cN$ by
\begin{equation}
	\label{Eq:HJB:ucN}
		\myGamma\ni \gamma_i^{\ast, \cN}(\mu) := \arg \max_{\gamma_i \in \myGamma} \{ L^{\cN,\gamma_i}(\mu; u^{\ast,\cN}(\mu)) - f^{\cN,\gamma_i}(\mu) \},
		\quad 1\le i\le\cN.
\end{equation}

\begin{rem}\label{Rem:Gamma}
	The above described model yields $\Gamma=L_\infty(\Omega_T; \R^d)$ or an appropriate subset. Such a subset would occur if control constraints would appear. In that case, $\R^d$ would be replaced by some $\Sigma\subset\R^d$. Correspondingly, $\R^{d\cN}$ in the discretization would have to be replaced by $\Sigma^\cN$. We emphasize that all subsequent findings remain true in this case.
\end{rem}

%
In order to simply notation, we collect all single optimization problems into one system by setting for $\gamma^\cN\in\myGammaN$, $u^\cN\in\R^\cN$
$$
	L^{\cN,\gamma^\cN}(\mu; u^\cN) := (L^{\cN,\gamma^\cN_i}(\mu; u^{\cN}))_{1\le i\le\cN}, \quad
	f^{\cN, \gamma^\cN}(\mu) := (f^{\cN,\gamma^\cN_i}(\mu))_{1\le i\le\cN},
$$
which means that $L^{\cN, \cdot}(\mu; \cdot): (\myGamma\times\R)^\cN \to\R^\cN$ and $f^{\cN, \cdot}(\mu):\myGammaN\to\R^\cN$. Hence, for the pair 
$$
	(\gamma^{\ast,\cN}(\mu),u^{\ast,\cN}(\mu)) \in\X^\cN := (\myGamma\times\R)^\cN = \R^{(d+1)\cN}
$$ 
of optimal (discrete) control and optimal (discrete) value function, we get that 
$$
	L^{\cN, \gamma^{\ast,\cN}(\mu)}(\mu; u^{\ast,\cN}(\mu)) - f^{\cN, \gamma^{\ast,\cN}(\mu)}(\mu)=0 \quad\text{ in } \R^\cN. 
$$

\subsection{A nonlinear system for control and state}
If the function $\gamma^\cN\mapsto L^{\cN, \gamma^\cN}(\mu, u^{\ast,\cN}(\mu)) - f^{\cN, \gamma^\cN}(\mu)$ is in $C^1(\myGamma; \R)^\cN$, we can consider the Fr\'echet derivative w.r.t.\ the control variable $\gamma$ at some $\delta^\cN\in\myGammaN$
$$
	\partial_\gamma [L^{\cN, \delta^\cN}(\mu; u^{\ast,\cN}(\mu))] \in L(\myGamma, \R)^\cN,
			\quad
	\partial_\gamma [f^{\cN, \delta^\cN}(\mu)] \in L(\myGamma, \R)^\cN.\footnotemark[1]
$$
\footnotetext[1]{We denote by $L(X,Y)$ the space of continuous, linear mappings from a normed space $X$ to a normed space $Y$.}%
Recall that in the case of the EU-ETS, $f^\gamma$ is a quadratic function of the control $\gamma$ so that the assumed differentiability in fact holds. 
Then, the optimal control $\gamma^{\ast,\cN}(\mu)\in\myGammaN$ is a critical point of this mapping, i.e.
\begin{equation}
	\label{Eq:Abl:u}
	\partial_\gamma [L^{\cN, \gamma^{\ast,\cN}(\mu)}(\mu; u^{\ast,\cN}(\mu))
	- f^{\cN, \gamma^{\ast,\cN}(\mu)}(\mu)](\delta^\cN) = 0,
	\quad \forall \delta^\cN\in\myGammaN,
\end{equation}
which is a \emph{nonlinear} problem for $\gamma^{\ast,\cN}(\mu)$ (even though \eqref{Eq:Abl:u} is linear in the `test function' $\delta^\cN$). 
We define the composite function $\cG^\cN(\mu): \X^\cN:=(\myGamma\times\R)^\cN \to L(\myGamma, \R)^\cN\times\R^\cN=:\Y^\cN$ as $(x^\cN=(\gamma^\cN, u^\cN)\in\X^\cN)$
\begin{equation}\label{Eq:def:G}
	\cG^\cN(\mu)(x^\cN) := 
			\begin{pmatrix} 
				\partial_\gamma [L^{\cN, \gamma^\cN}(\mu; u^\cN) - f^{\cN, \gamma^\cN}(\mu)] \\[1mm]
				L^{\cN,\gamma^\cN}(\mu; u^\cN) - f^{\cN,\gamma^\cN}(\mu) 
			\end{pmatrix}
			=:
			\begin{pmatrix} 
				\cG_1^\cN(\mu) (x^\cN) \\[2mm] 
				\cG_2^\cN(\mu) (x^\cN) 	
			\end{pmatrix}.
\end{equation}
In this notation, the `truth' control/state-solution $x^{\cN,\ast}(\mu):=(\gamma^{\ast,\cN}(\mu), u^{\ast,\cN}(\mu))\in\X^\cN$ is characterized by
\begin{equation}\label{Eq:truthproblem}
	\cG^\cN(\mu) (x^{\cN,\ast}(\mu)) =
	\cG^\cN(\mu) (\gamma^{\ast,\cN}(\mu), u^{\ast,\cN}(\mu)) = 0 \quad\text{in } \Y^\cN,
\end{equation}
i.e., this equation is to be understood in $L(\myGamma,\R)^\cN\times\R^\cN = \Y^\cN$.

For later purpose, we determine the Fr\'echet derivatives of $\cG^\cN(\mu)$, in case of their existence (which is obviously guaranteed for the EU-ETS case), of course. Let $x^\cN=(\gamma^\cN,u^\cN)\in\X^\cN$. Then, we have $D(\cG^\cN(\mu))(x^\cN)\in L(\X^\cN, \Y^\cN)$, i.e.,  $(D(\cG^\cN(\mu))(x^\cN))(\tilde x^\cN)\in \Y^\cN$ for $\tilde x^\cN=(\tilde\gamma^\cN,\tilde u^\cN)\in\X^\cN$ and get
\begin{align}
	(D(\cG^\cN(\mu))(x^\cN))(\tilde x^\cN)
	&=
	\begin{pmatrix}
		\partial_\gamma(\cG_1^\cN(\mu)(x^\cN))(\tilde\gamma^\cN) 
				+ \partial_u(\cG_1^\cN(\mu)(x^\cN))(\tilde{u}^\cN) \\[1mm]
		\partial_\gamma(\cG_2^\cN(\mu)(x^\cN))(\tilde\gamma^\cN) 
				+ \partial_u(\cG_2^\cN(\mu)(x^\cN))(\tilde{u}^\cN) 
	\end{pmatrix}
	\notag
	\\
	&\kern-100pt=
	\begin{pmatrix}
		\partial_\gamma(\partial_\gamma [L^{\cN, \gamma^\cN}\kern-3pt(\mu; u^\cN) 
			\kern-2pt-\kern-2pt f^{\cN, \gamma^\cN}(\mu)] )(\tilde\gamma^\cN) 
				+ \partial_u(\partial_\gamma [L^{\cN, \gamma^\cN}\kern-3pt(\mu; u^\cN) 
				\kern-2pt-\kern-2pt f^{\cN, \gamma^\cN}(\mu)] )(\tilde{u}^\cN) \\[1mm]
		\partial_\gamma(L^{\cN,\gamma^\cN}(\mu; u^\cN) - f^{\cN,\gamma^\cN}(\mu))(\tilde\gamma^\cN) 
				+ \partial_u(L^{\cN,\gamma^\cN}(\mu; u^\cN) - f^{\cN,\gamma^\cN}(\mu))(\tilde{u}^\cN) 
	\end{pmatrix}
	\notag
	\\
	&\kern-2pt=
	\begin{pmatrix}
		\partial_\gamma^2 [ L^{\cN, \gamma^\cN}(\mu; u^\cN) -f^{\cN, \gamma^\cN}(\mu)](\tilde\gamma^\cN) 
			+ \partial_\gamma [L^{\cN, \gamma^\cN}(\mu; \tilde u^\cN)] \\[1mm]
		\partial_\gamma [ L^{\cN, \gamma^\cN}(\mu; u^\cN) -f^{\cN, \gamma^\cN}(\mu)](\tilde\gamma^\cN)
			+ L^{\cN, \gamma^\cN}(\mu; \tilde u^\cN)
	\end{pmatrix}. 
	\label{eq:DG}
\end{align}

\begin{rem}\label{Rem:2.2}
One possible numerical method to determine a solution of \eqref{Eq:HJB:cN} is the so-called \emph{policy iteration algorithm}, also called \emph{Howard's algorithm}, \cite{MR2551155}, which reads as follows for an initial value $u^{(0)}\in\R^\cN$: For $k=0,1,2,\ldots$ do
\begin{subequations}
	\label{Eq:Howard}
	\begin{align}
		\myGammaN\ni \gamma^{(k+1)} 
			&= \arg \max_{\gamma \in \myGammaN} \{ L^{\cN,\gamma}(\mu; u^{(k)}) - f^{\cN,\gamma}(\mu)\}, \label{Eq:Howard:1} \\
		\text{find } u^{(k+1)}\in\R^\cN:\quad &   L^{\cN, \gamma^{(k+1)}}(\mu; u^{(k+1)}) = f^{\cN, \gamma^{(k+1)}}(\mu). \label{Eq:Howard:2} 
	\end{align}
\end{subequations}
Under appropriate conditions, this algorithm converges and for the limits, we have $u^{(k)}\to u^{\ast,\cN}(\mu)$ (the solution of \eqref{Eq:HJB:cN}) $\gamma^{(k)} \to \gamma^{\ast,\cN}(\mu)$ as $k\to\infty$ (defined in \eqref{Eq:HJB:ucN}). 
\end{rem}

For later use in deriving error estimates, we collect conditions that ensure Lip\-schitz continuity of the Fr\'echet derivative $D\cG^\cN(\mu)$.

\begin{lemma}
\label{Lemma:Lipschitz}
If the estimates
\begin{subequations}
	\begin{align}
		\| L^{\cN, \gamma_1^\cN} - L^{\cN, \gamma_2^\cN}\|_{L(\R^\cN,\R^\cN)}
			&\le \varrho^L_0 \| \gamma_1^\cN - \gamma_2^\cN\|_{\myGammaN}, 
				\label{eq:L0} \\
		\| \partial_\gamma(L^{\cN, \gamma_1^\cN}u^\cN_1 
			- L^{\cN, \gamma_2^\cN} u^\cN_2)\|_{L(\myGammaN,\R^\cN)}
			&\le \varrho^L_1 \| x^\cN_1 - x^\cN_2\|_{\X^\cN}, 
				\label{eq:L1}\\		
		\| \partial_\gamma^2(L^{\cN, \gamma^\cN_1}u_1^\cN 
			- L^{\cN, \gamma^\cN_2} u_2^\cN)\|_{L(\myGammaN,L(\myGammaN,\R^\cN))}
			&\le \varrho^L_2 \| x^\cN_1 - x^\cN_2\|_{\X^\cN}, 	
				\label{eq:L2}
	\end{align}
\end{subequations}
hold for constants $\varrho^L_k<\infty$, $k=0,1,2$, $x_i=(\gamma_i, u_i)\in\X^\cN$, $i=1,2$ and 
\begin{subequations}
	\begin{align}
		\| \partial_\gamma (f^{\cN, \gamma^\cN_1}-f^{\cN, \gamma^\cN_2})\|_{L(\myGammaN,\R^\cN)}
			&\le \varrho^f_1 \| \gamma_1^\cN-\gamma_2^\cN\|_{\myGammaN}, 
				\label{eq:f1}\\
		\| \partial_\gamma^2 (f^{\cN,\gamma^\cN_1}-f^{\cN, \gamma^\cN_2})\|_{L(\myGammaN,L(\myGammaN,\R^\cN))}
			&\le \varrho^f_2 \| \gamma^\cN_1-\gamma^\cN_2\|_{\myGammaN},
				\label{eq:f2}
	\end{align}
\end{subequations}
	for constants $\varrho^f_k<\infty$, $k=1,2$, then $D(\cG^\cN(\mu))$ is Lipschitz continuous, i.e.,
	\begin{equation}\label{eq:DGLip}
		\| D(\cG^\cN(\mu))(x^\cN_1) - D(\cG^\cN(\mu))(x^\cN_2)\|_{L(\X^\cN,\Y^\cN)}
			\le \varrho \| x^\cN_1-x^\cN_2\|_\X
	\end{equation}
	with $\varrho = \varrho^L_0 + 2\varrho^L_1 + \varrho^L_2 + \varrho^f_1 + \varrho^f_2$. 
\end{lemma}
\begin{proof}
	The proof is more or less standard using the representation \eqref{eq:DG} of $D(\cG^\cN(\mu))$ and triangle inequalities several times. It only remains to note that using $u^\cN=u_1^\cN=u_2^\cN$ in \eqref{eq:L1} implies 
	$
		\| \partial_\gamma(L^{\cN, \gamma_1^\cN}u^\cN 
			- L^{\cN, \gamma_2^\cN} u^\cN)\|_{L(\myGammaN,\R^\cN)}
			\le \varrho^L_1 \| \gamma^\cN_1 - \gamma^\cN_2\|_{\X^\cN}
			$,
	which results in the estimate $\| \partial_\gamma(L^{\cN, \gamma_1^\cN}
			- L^{\cN, \gamma_2^\cN})\|_{L(\R^\cN,L(\myGammaN,\R^\cN))}
			\le \varrho^L_1 \| \gamma^\cN_1 - \gamma^\cN_2\|_{\X^\cN}$,
		which is needed in the estimate \eqref{eq:DGLip}.
\end{proof}

\begin{rem}
	Recall that all assumptions in Lemma \ref{Lemma:Lipschitz} are satisfied for the EU-ETS. In fact, the operator $L^\gamma$ is linear in $\gamma$ and is a quadratic function of $\gamma$.
\end{rem}

\subsection{Well-posedness and error control}
We now present some well-known ingredients of the theory developed by Brezzi, Rappaz and Raviart, known as \emph{BRR theory}, \cite{BRR1,BRR2,BRR3}, which provides us with results concerning existence and uniqueness as well as error control for the nonlinear system \eqref{Eq:truthproblem}. In order to streamline the notation, we consider a general generic framework. To this end, let $X$, $Y$ be two finite-dimensional spaces normed by $\|\cdot\|_X$ and $\|\cdot\|_Y$, respectively. We consider a nonlinear mapping $G:X\to Y$ and seek for a solution $x^\ast\in X$ of the problem $G(x)=0$ in $Y$ (i.e., a general form of \eqref{Eq:truthproblem}). 

Next, we assume that the Fr\'echet derivative $DG(x)\in L(X,Y)$ exists for all $x\in X$ and that the inverse also exists. Then, we define for some fixed $\bar{x}\in X$ the mapping
\begin{equation}\label{eq:DefH}
	H_{\bar{x}}: X\to X,
	\qquad
	H_{\bar{x}} (x) := x - (DG(\bar{x}))^{-1} G(x),
\end{equation}
which is obviously of quasi-Newton type. Finally, we set
\begin{equation}\label{eq:Defbeta}
	\beta_{\bar{x}} := \| (DG(\bar{x}))^{-1}\|^{-1}_{L(Y,X)}
\end{equation}
and assume that $0<\beta_{\bar{x}}<\infty$.

\begin{lemma}\label{Lem:2.3}
	Let the Fr\'echet derivative $DG(x)\in L(X,Y)$ exist for all $x\in X$, be invertible and Lipschitz continuous, i.e., there exists a constant $\varrho >0$ such that
	\begin{equation}\label{eq:ass:DG}
		\| DG(x_1) - DG(x_2)\|_{L(X,Y)} \le \varrho \| x_1-x_2\|_X
	\end{equation}
	for all $x_1, x_2\in X$
	Then, for any $x_1, x_2\in X$, we have
\begin{align*}
	\| H_{\bar{x}}(x_1) &- H_{\bar{x}}(x_2)\|_{X}
	\le  
 	\frac{\varrho}{\beta_{\bar{x}}} \, \|x_1-x_2\|_{X}
	\int_0^1  \| {\bar{x}} - (x_2+t(x_1-x_2))\|_{X}\, dt.
\end{align*}
\end{lemma}
\begin{proof}
	The proof follows standard lines starting with the fundamental theorem of calculus
	$G(x_1) = G(x_2)  + \int_0^1 DG\big( x_2+t(x_1-x_2)\big) (x_1-x_2)\, dt$. 
	Then, we get
	\begin{align*}
	H_{\bar{x}}(x_1) - H_{\bar{x}}(x_2)
	&= x_1 - x_2 - \big( DG({\bar{x}})\big)^{-1}\big( G(x_1)-G(x_2)\big) \\
	&\kern-40pt= \big( DG({\bar{x}})\big)^{-1} 
	\big\{ \big( DG({\bar{x}})\big)(x_1-x_2) - \big( G(x_1)-G(x_2) \big)\big\} \\
	&\kern-40pt= \big( DG({\bar{x}})\big)^{-1} 
	\Big\{ \big( DG({\bar{x}})\big)(x_1-x_2) 
		- \int_0^1 DG\big(x_2 + t(x_1-x_2)\big) (x_1-x_2)\, dt\Big\} \\
	&\kern-40pt= \big( DG({\bar{x}})\big)^{-1} 
		\int_0^1 \big\{ DG({\bar{x}}) - DG	\big(x_2 + t(x_1-x_2)\big)\big\} (x_1-x_2)\, dt.
	\end{align*}
	Next, we use standard estimates to obtain 
	\begin{align*}
	\| H_{\bar{x}}(x_1) - H_{\bar{x}}(x_2) \|_{X}
		&\le \frac1{\beta_{\bar{x}}}	\,
		\int_0^1 \big\| \big\{ DG({\bar{x}}) - DG	 \big(x_2 + t(x_1-x_2)\big)\big\} (x_1-x_2)\big\|_{Y} dt \\
		&\le \frac{\varrho}{\beta_{\bar{x}}}  \| x_1-x_2\|_{X}
		\int_0^1 \| {\bar{x}} - (x_2+t(x_1-x_2))\|_{X}\, dt,
	\end{align*}
	by \eqref{eq:ass:DG}, which proves the claim.
\end{proof}
\medskip

\begin{lemma}\label{Lem:2.4}
	Let the assumptions of Lemma \ref{Lem:2.3} hold. Then, the mapping $H_{\bar{x}}$ is a contraction on $\bar{B}_\gamma({\bar{x}}):=\{ x\in X:\, \| {\bar{x}}-x\|_{X}\le\gamma\}$ if $\gamma < \gamma_{\mathrm{contr.}}:=\frac{\beta_{\bar{x}}}{\varrho}$.
\end{lemma}
\begin{proof}
	Let $x_1, x_2\in \bar{B}_\gamma({\bar{x}})$. Then, we have $ \| {\bar{x}} - (x_2+t(x_1-x_2))\|_{X} \le\gamma$ and the assertion follows immediately by Lemma \ref{Lem:2.3}.
\end{proof}
\medskip

\begin{lemma}\label{Lem:2.5}
	Let the assumptions of Lemma \ref{Lem:2.3} hold. Then, the mapping $H_{\bar{x}}: \bar{B}_\gamma(\bar{x})\to \bar{B}_\gamma(\bar{x})$ is a self-map for all $\bar{x}\in X$ with
		\begin{equation}\label{eq:indicator}
		\tau(\bar{x}) := \frac{2\varrho}{\beta_{\bar{x}}^2} \| G(\bar{x})\|_Y \le 1,
		\end{equation}
	and 
			\begin{equation}\label{eq:delta}
				\gamma \in [\gamma_{\mathrm{min}}, \gamma_{\mathrm{max}}]
				:= \frac{\beta_{\bar{x}}}{\varrho}
					\Big[ 1-\sqrt{1-\tau(\bar{x})},
						1+\sqrt{1-\tau(\bar{x})} \Big].	
		\end{equation}
\end{lemma}
\begin{proof}
	We start by the simple identity 
	$
		H_{\bar{x}}(x) - \bar{x}
		= H_{\bar{x}}(x) - H_{\bar{x}}(\bar{x}) + H_{\bar{x}}(\bar{x}) - \bar{x}
		= H_{\bar{x}}(x) - H_{\bar{x}}(\bar{x}) - (DG(\bar{x}))^{-1} G(\bar{x}), 
	$
	which holds for for any $x\in X$. 
	Then, 
	\begin{align*}
	\| H_{\bar{x}}(x) - \bar{x}\|_X
	&\le \| H_{\bar{x}}(x) - H_{\bar{x}}(\bar{x}) \|_X + \beta_{\bar{x}}^{-1} \| G(\bar{x})\|_Y.
	\end{align*}
	If $x\in \bar{B}_\gamma(\bar{x})$, we can further estimate the first term by Lemma \ref{Lem:2.3} and obtain
	\begin{align*}
	\| H_{\bar{x}}(x) - \bar{x}\|_X 
	& \le \frac{\varrho \gamma}{\beta_{\bar{x}}} \int_0^1 t \| \bar{x} - x\|_X\, dt + \beta_{\bar{x}}^{-1} \| G(\bar{x})\|_Y
	\le \frac{\varrho \gamma^2}{2\, \beta_{\bar{x}}} + \beta_{\bar{x}}^{-1} \| G(\bar{x})\|_Y.
	\end{align*}
	This latter term is less than $\gamma$ if and only if $\gamma^2 - \frac{2\beta_{\bar{x}}}{\varrho} \gamma + \frac{2}{\varrho} \| G(\bar{x})\|_Y\le 0$, which in turn is valid for $\gamma \in [\gamma_{\mathrm{min}}, \gamma_{\mathrm{max}}]$ and $\| G(\bar{x})\|_Y \le (2\, \varrho)^{-1} {\beta_{\bar{x}}^2}$.
\end{proof}
\medskip

Summarizing the above findings, we get the following result.

\begin{proposition}\label{Prop:2.1}
	Let the Fr\'echet derivative $DG(x)\in L(X,Y)$ exist for all $x\in X$, be invertible and Lipschitz continuous with constant $\varrho$. Let $\bar{x}\in X$ be given such that \eqref{eq:indicator} holds.
Then, there exists a unique fixed-point $x^\ast\in \bar{B}_\gamma(\bar{x})$ of $H_{\bar{x}}$ for all $\gamma\in [\gamma_{\mathrm{min}},\gamma_{\mathrm{contr.}})$.
\end{proposition}

\begin{proof}
	The proof follows from Banach fixed-point theorem in view of Lemma \ref{Lem:2.3} and \ref{Lem:2.5}. 
\end{proof}
\medskip

Proposition \ref{Prop:2.1} yields a well-posedness result, but at the same time also provides us with an error estimate as we shall see next.

\begin{corollary}\label{Cor:1}
	Under the assumptions of Proposition \ref{Prop:2.1}, the estimate
	\begin{equation}\label{eq:Est}
		\| x^\ast - \bar{x}\|_X \le  \frac{\beta_{\bar{x}}}{\varrho} (1-\sqrt{1- \tau(\bar{x})})
	\end{equation}
	holds for $\bar{x}\in X$ satisfying \eqref{eq:indicator}.
\end{corollary}
\begin{proof}
	The mapping $H_{\bar{x}}$ has a unique fixed-point $x^\ast$ in $\bar{B}_\gamma(\bar{x})$ for $\gamma=\gamma_{\mathrm{min}}$.
\end{proof}
\medskip

\begin{rem} The following observations are potentially crucial for the numerical realization:
	\begin{compactitem}
	\item[(a)]
	Note, that	the quantity $\tau(\bar{x})$ is an a posterori indicator provided that $\varrho$ (or an estimate) is available and $\beta_{\bar{x}}$ (or an estimate) is computable. In fact, $\| G(\bar{x})\|_Y$ is the computable residual of the nonlinear equation and thus \eqref{eq:indicator} can be verified a posteriori. We will later use this observation to obtain an error estimate for a numerical approximation $\bar{x}$ of the solution of the nonlinear problem $G(x)=0$.
	\item[(b)] 
	Obviously, the assumptions on the Fr\'echet derivative $DG$ only need to hold in a neighborhood of $\bar{x}$.
	\end{compactitem}
\end{rem}

\section{The Reduced Basis Method (RBM)}\label{Sec:4}

We now introduce the Reduced Basis Method (RBM) for the numerical solution of the parameterized HJB equation and start by reviewing the standard RB setting. The main idea is to select in an offline phase so-called snapshot parameters
$$
	S_N := \{ \mu_1, \ldots , \mu_N\},
$$
and compute the corresponding snapshots $x_i:= x^{\ast,\cN}(\mu_i)=(\gamma^{\ast,\cN}(\mu_i), u^{\ast,\cN}(\mu_i))\in\X^\cN$ as the solution of \eqref{Eq:HJB:cN}, \eqref{Eq:HJB:ucN}, or in other terms, \eqref{Eq:truthproblem} in $\Y^\cN$. The precise selection of the snapshots will be explained later. Then, we define the RB space as $X_N:=\mathrm{span}\{ x_i^\cN:\, i=1, \ldots, N\}$, assuming that $N\ll\cN$.

An RB approximation $x_N^\ast(\mu)\in X_N$ of $x^{\ast,\cN}(\mu)\in\X^\cN$, $\mu\in\cD\setminus S_N$,  is then computed by the solution of
$$
	\cG(\mu) (x_N^\ast(\mu)) = 0 \quad\mathrm{in }\, Y_N(\mu),
$$
where $Y_N(\mu)\subset\Y^\cN$ is some test space which is possibly $\mu$-dependent and such that the reduced problem is stable. 

The aim is to use the RBB-theory for developing an a posteriori error estimate for the RB-approximation. To this end, we fix some ${\bar{x}}\in\X^\cN$ (to be determined below) and define for any $x^\cN=(\gamma^\cN,u^\cN)\in\X^\cN$ the mapping
\begin{equation}\label{eq:DefHw}
	H_{\bar{x}}(\mu) (x^\cN) := x^\cN - ((D\cG^\cN(\mu)) ({\bar{x}^\cN}))^{-1} \cG^\cN(\mu) (x^\cN),
	\,\,\,
	H_{\bar{x}}(\mu):  \X^\cN \to \X^\cN.
\end{equation}
The idea is to use $\bar{x} = x_N^\ast(\mu) = (\gamma_N^\ast(\mu), u_N^\ast(\mu))$ and set analogously to \eqref{eq:Defbeta}
\begin{align}
	\beta_N (\mu) := \beta_{x_N^\ast(\mu)}(\mu) 
		&:= \| ((D\cG^\cN(\mu)) (x_N^\ast(\mu) ))^{-1}\|^{-1}_{L(\Y^\cN,\X^\cN)} \label{eq:def:beta} \\
		&= \inf_{x^\cN\in \X^\cN} \frac{\| (D\cG^\cN(\mu) (x_N^\ast(\mu) ))(x^\cN)\|_{\Y^\cN} }{\|x^\cN\|_{\X^\cN}}.
		\notag
\end{align}

We mention \cite{Alla,AllaHinze} for a POD-based model reduction approach of the HJB-equation, where the reduction is performed with respect to time.

\subsection{Computation of a lower inf-sup bound}
To obtain a lower bound for the inf-sup constant $\beta_N(\mu)$, we follow an idea presented in \cite{Masa}. 
We start by detailing the computation of a lower bound for $\beta_N (\mu) $. Since $\Y^\cN=L(\myGamma, \R)^\cN\times\R^\cN$ is a Hilbert space, we get
\begin{align*}
	\beta_N (\mu) 
		&= \inf_{x^\cN\in \X^\cN} \sup_{\tilde y^\cN \in \Y^\cN} 
			\frac{\langle (D\cG^\cN(\mu) (x_N(\mu)))(x^\cN), \tilde y^\cN \rangle_{\Y^\cN}}{\|x^\cN\|_{\X^\cN} \|\tilde y^\cN\|_{\Y^\cN}}.
\end{align*}
For any given $\bar{x}, x\in \X^\cN$, we define the \emph{supremizer} $s(\mu; \bar{x},x)\in\Y^\cN$ as
\begin{equation}\label{Eq:Supremizer}
	s(\mu; \bar{x},x) :=
		\arg \sup_{\tilde y^\cN \in \Y^\cN} 
			\frac{\langle (D\cG^\cN(\mu) (\bar{x}))(x), \tilde y^\cN \rangle_{\Y^\cN}}{\|x\|_{\X^\cN} \|\tilde y^\cN\|_{\Y^\cN}},
\end{equation}
so that
\begin{align*}
	\beta_N (\mu) 
		&= \inf_{x^\cN\in \X^\cN}  
			\frac{\langle (D\cG^\cN(\mu) (x_N^\ast(\mu)))(x^\cN), s(\mu; x_N^\ast(\mu), x^\cN) \rangle_{\Y^\cN}}{\|x^\cN\|_{\X^\cN}\, \|s(\mu; x_N^\ast(\mu), x^\cN)\|_{\Y^\cN}} \\
		&\ge \inf_{x^\cN\in \X^\cN}  
			\frac{\langle (D\cG^\cN(\mu) (x^\ast_N(\mu)))(x^\cN), \tilde s(\mu; x^\cN) \rangle_{\Y^\cN}}{\|x^\cN\|_{\X^\cN}\, \| \tilde s(\mu; x^\cN)\|_{\Y^\cN}},
\end{align*}
where $\tilde s(\mu; x^\cN)\in\Y^\cN$ is arbitrary and will be chosen later. Then, we obtain
\begin{align}
	\beta_N (\mu) 
		&\ge 
			\bigg[\inf_{x^\cN\in \X^\cN}\kern-5pt  \frac{\| \tilde s(\mu; x^\cN)\|_{\Y^\cN}}{\|x^\cN\|_{\X^\cN}}\bigg]
			\bigg[
			\inf_{x^\cN\in \X^\cN}\kern-5pt
			\frac{\langle (D\cG^\cN(\mu) (x_N(\mu)))(x^\cN), \tilde s(\mu; x^\cN)\rangle_{\Y^\cN}}{\| \tilde s(\mu; x^\cN)\|_{\Y^\cN}^2}\bigg]
				\notag \\
		&=: \beta^{\textrm{offline}}_{\textrm{LB}}(\mu)\, \beta^{\textrm{online}}_{\textrm{LB}}(\mu),
		\label{eq:betaLB:1}
\end{align}
where $\beta^{\textrm{offline}}_{\textrm{LB}}(\mu)$, $\beta^{\textrm{online}}_{\textrm{LB}}(\mu)$ are lower bounds to be computed offline and online, respectively. 

We determine offline a small set of $R\in\N$ so called \emph{anchor} parameters $\bar{S}_R := \{\bar\mu_1, \ldots , \bar\mu_R\}\subset\cD$.
Online, for a given $\mu\in\cD$, we determine the `closest' anchor parameter $\bar\mu(\mu)\in \bar{S}_R$ defined by minimizing $|\mu - \bar\mu|$  over $\bar\mu\in \bar{S}_R$. Since we choose $\tilde s(\mu; x^\cN) := s(\bar\mu(\mu); x_N^\ast(\mu), x^\cN)$, it holds $\beta^{\textrm{offline}}_{\textrm{LB}}(\mu) = \beta_N(\bar\mu(\mu))$. Thus, in the offline-phase $\beta_N(\bar\mu_r)$, $r=1, \ldots, R$, are precomputed solving generalized eigenvalue problems. 
To determine the anchor points $\bar\mu_r$, we use a Greedy-type method detailed in Algorithm \ref{Alg:infsup} based upon a \emph{training set} $\Xitrainanchor\subset\cD$. Of course, the constant $\frac12$ in line \ref{alg:12} of Algorithm \ref{Alg:infsup} could be verified. This choice is motivated by \eqref{eq:betaLB} since Algorithm \ref{Alg:infsup} ensures that
$
	\min\limits_{\mu\in\Xitrainanchor} \beta^{\textrm{online}}_{\textrm{LB}}(\mu) > \frac12,
$
whose relevance will be described next.

\begin{algorithm}\caption{Greedy selection of anchor points.}\label{Alg:infsup}
	\begin{algorithmic}[1]
		\STATE choose $\bar\mu_1 \in \cD$ arbitrarily, $N \leftarrow 1$, $\bar{S}_R := \{\bar\mu_1\}$, compute $\beta^{\textrm{offline}}_{\textrm{LB}}(\bar\mu_{1})$
		\WHILE{$\min\limits_{\mu\in\Xitrainanchor} \beta^{\textrm{online}}_{\textrm{LB}}(\mu) \leqslant \frac12$} \label{alg:12}
			\STATE $\bar\mu_{N+1} \leftarrow \mathop{\arg\min}\limits_{\mu\in\Xitrainanchor} \beta^{\textrm{online}}_{\textrm{LB}}(\mu)$ 
			\STATE $\bar{S}_R \leftarrow  \bar{S}_R\cup\{\bar\mu_{N+1}\}$, compute $\beta^{\textrm{offline}}_{\textrm{LB}}(\bar\mu_{N+1})$
			\STATE $N\leftarrow N+1$
		\ENDWHILE
	\end{algorithmic}	
\end{algorithm}

The online part  $\beta^{\textrm{online}}_{\textrm{LB}}(\mu)$ of the bound relies on a separation of $D\cG^\cN(\mu)$ with respect to the parameter and is computed by the well-known Successive Constraint Method (SCM) \cite{SCM}, which will be discussed below.

Altogether, we obtain an inf-sup lower bound
\begin{equation}\label{eq:betaLB}
	\beta^{\text{LB}}_N(\mu) 
		:= \beta^{\textrm{offline}}_{\textrm{LB}}(\bar\mu(\mu))\, \beta^{\textrm{online}}_{\textrm{LB}}(\mu)
		> \frac12 \beta^{\textrm{offline}}_{\textrm{LB}}(\mu).
\end{equation}
In order to compute an upper bound for the indicator in \eqref{eq:indicator}, we define
\begin{equation}\label{eq:indicator:RB}
	\tau_N(\mu) := \tau(x_N^\ast(\mu)) = \frac{2\varrho}{(\beta_N(\mu))^2}\, \| \cG^\cN(\mu)(x_N^\ast(\mu))\|_{\Y^\cN}
\end{equation}
and the corresponding upper bound
\begin{equation}\label{eq:indicator:RBUB}
	\tau_N^{\textrm{UB}}(\mu) := \frac{2\varrho}{(\beta_N^{\textrm{LB}}(\mu))^2}\, \| \cG^\cN(\mu)(x_N^\ast(\mu))\|_{\Y^\cN}.
\end{equation}
Using \eqref{eq:Est}, the error can thus be estimated with
\begin{equation}\label{eq:estimator:RB}
	\Delta_N(\mu) := \frac{\beta_N^{\textrm{LB}}(\mu)}{\varrho} \Big( 1-\sqrt{1-\tau_N^{\textrm{UB}}(\mu)}\Big).
\end{equation}

\subsection{Offline/online-separation}
For an efficient separation of the required calculations into a possibly costly offline and a highly efficient online phase, one usually requires a separation of the parameter and other types of variables, sometimes also called affine decomposition. Here, this means
	\label{Eq:paramsep}
	\begin{align}
		L^\gamma(\mu; u) 
			&= \sum_{q=1}^{Q_L} \vartheta_q^L(\mu)\, L^\gamma_q(u), 
		&
		f^\gamma(\mu) 
			&= \sum_{q=1}^{Q_f} \vartheta_q^f(\mu)\, f^\gamma_q,
	\end{align}
with functions $\vartheta_q^L, \vartheta_q^f:\cD\to\R$. 
This separability in the parameter also transfers to their discretized variants $L^{\cN,\gamma^\cN}$, $f^{\cN,\gamma^\cN}$ and to $D\cG$ so that \eqref{eq:DG} reads as
\begin{align}
	(D(\cG^\cN(\mu))(x^\cN))(\tilde x^\cN) &= &&
		\sum_{q=1}^{Q_L} \vartheta_q^{L^\cN}(\mu) \begin{pmatrix}
			\partial_\gamma^2 [ L_q^{\cN,\gamma^\cN}(u^\cN) ] (\tilde \gamma^\cN) + \partial_\gamma [ L_q^{\cN,\gamma^\cN}(\tilde u^\cN) ] \\
			\partial_\gamma [ L_q^{\cN,\gamma^\cN}(u^\cN) ] (\tilde \gamma^\cN) + L_q^{\cN,\gamma^\cN}(\tilde u^\cN) \end{pmatrix} \notag \\ &&&
		- \sum_{q=1}^{Q_f} \vartheta_q^{f^\cN}(\mu) \begin{pmatrix}
			\partial_\gamma^2 [ f_q^{\cN,\gamma^\cN} ] (\tilde \gamma^\cN) \\
			\partial_\gamma [ f_q^{\cN,\gamma^\cN} ] (\tilde \gamma^\cN) \end{pmatrix} \notag \\ &
		=: && \sum_{q=1}^{Q_L} \vartheta_q^{L^\cN}(\mu) \, D\cG_q^L(x^\cN)(\tilde x^\cN)
    		- \sum_{q=1}^{Q_f} \vartheta_q^{f^\cN}(\mu) \, D\cG_q^f(\tilde x^\cN).  \label{Eq:DG:paramsep}
\end{align}
By inserting \eqref{Eq:DG:paramsep} and the representation $x_N(\mu) = \sum_{n=1}^N \bx_N^n(\mu) \, \xi_n$ in the definition \eqref{eq:betaLB:1} of $\beta^{\textrm{online}}_{\textrm{LB}}(\mu)$ we obtain:
\begin{subequations}
\begin{align}
	\beta^{\textrm{online}}_{\textrm{LB}}(\mu) = \inf_{x^\cN\in\X^\cN} \Bigg[ &
		 \frac{\langle (D\cG^\cN(\bar\mu) (x_N(\bar\mu)))(x^\cN), \tilde s(\mu; x^\cN)\rangle_{\Y^\cN}}{\| \tilde s(\mu; x^\cN)\|_{\Y^\cN}^2}  \notag \\ & \kern-35pt
		+\sum_{q=1}^{Q_L} \sum_{n=1}^N \left( \vartheta_q^{L^\cN}(\mu)\bx_N^n(\mu) - \vartheta_q^{L^\cN}(\bar\mu)\bx_N^n(\bar\mu) \right)
			\frac{\langle D\cG_q^L(\xi_n)(x^\cN) , \tilde s(\mu; x^\cN)\rangle_{\Y^\cN} }{\| \tilde s(\mu; x^\cN)\|_{\Y^\cN}^2}  \label{Eq:betaonline:1} \\&
		\hspace*{-12mm}
		+\sum_{q=1}^{Q_f} \left( \vartheta_q^{f^\cN}(\mu) - \vartheta_q^{f^\cN}(\bar\mu) \right)
			\frac{\langle D\cG_q^f(\tilde x^\cN) , \tilde s(\mu; x^\cN)\rangle_{\Y^\cN} }{\| \tilde s(\mu; x^\cN)\|_{\Y^\cN}^2}   \Bigg]. 
			\label{Eq:betaonline:2}
\end{align}
\end{subequations}
In order to apply the Successive Constraint Method, we summarize the terms in \eqref{Eq:betaonline:1} and \eqref{Eq:betaonline:2} in $\cT(\mu; z)$
\begin{align*}
	\cT(\mu; z) := &
		\sum_{q=1}^{Q_L} \sum_{n=1}^N \left( \vartheta_q^{L^\cN}(\mu)\bx_N^n(\mu) - \vartheta_q^{L^\cN}(\bar\mu)\bx_N^n(\bar\mu) \right) z_{n+(q-1)N} \\&+
		\sum_{q=1}^{Q_f} \left( \vartheta_q^{f^\cN}(\mu) - \vartheta_q^{f^\cN}(\bar\mu) \right) z_{Q_LN+q}.
\end{align*}
Since we have chosen $\tilde s(\mu; x^\cN) = (D\cG^\cN(\bar\mu) (x_N(\bar\mu)))(x^\cN)$, we get the representation 
$\beta^{\textrm{online}}_{\textrm{LB}}(\mu) = 1 + \displaystyle{\inf_{z\in Z_N} \cT(\mu;z)}$ with 
\begin{align*}
	Z_N = \Bigg\{ z \in \R^{Q_LN+Q_f} : \exists \, x \in \X^\cN, \, &
		z_{n+(q-1)N} = \frac{\langle D\cG_q^L(\xi_n)(x^\cN) , \tilde s(\mu; x^\cN)\rangle_{\Y^\cN} }{\| \tilde s(\mu; x^\cN)\|_{\Y^\cN}^2}, \\ &
		z_{Q_LN+q} = \frac{\langle D\cG_q^f(\tilde x^\cN) , \tilde s(\mu; x^\cN)\rangle_{\Y^\cN} }{\| \tilde s(\mu; x^\cN)\|_{\Y^\cN}^2} & \Bigg\}.
\end{align*}
To obtain a lower bound for $\inf_{z\in Z_N} \cT_N(\mu;z)$ we use the Successive Constraint Method from \cite{SCM} and use an appropriate superset of $Z_N$.

Next, we need the Lipschitz-constant $\varrho$ in \eqref{eq:DGLip}, which in the general case could also be $\mu$-dependent. In the simpler case $\varrho\ne\varrho(\mu)$, this constant can be computed up to numerical precision by solving a generalized eigenvalue problem offline.

Finally, it remains to compute the residual $ \| \cG^\cN(\mu)(x_N^\ast(\mu))\|_Y$, which depends on the problem at hand and also requires the above mentioned separation. We will detail this for a specific numerical example in Section \ref{Sec:4} below. Combining all this, the inf-sup lower bound $\beta^{\text{LB}}_N(\mu)$ in \eqref{eq:betaLB}, the indicator $\tau_N(\mu)$ in \eqref{eq:indicator:RB} and the error estimator $\Delta_N(\mu)$ can be computed online-efficient.

\subsection{Greedy Algorithm}\label{sec_greedy}
Now, we describe the computation of the snapshot parameters $S_N$. To this end, we use a standard Greedy method over the error estimator $\Delta_N(\mu)$ based upon a training set $\Xitrain\subset\cD$. The scheme is displayed in Algorithm \ref{Alg:Greedy}. 
As we have seen, the error estimator $\Delta_N(\mu)$ is only meaningful if the indicator $\tau_N(\mu)$ is less than one. In order to ensure this, we perform a 2-stage Greedy similar to e.g.\ \cite{Tonn,Deparis,Veroy1,Veroy2,Veroy3}. 
In the first step, we determine a preliminary snapshot set $S_M$ such that $\tau_M(\mu)<1$ for all $\mu\in\Xitrain$. This is realized in lines \ref{alg:2.2}-\ref{alg:2.3} in Algorithm \ref{Alg:Greedy}. The second loop in lines \ref{alg:2.4}-\ref{alg:2.5} enriches the so determined $S_M$ so that for the resulting set $S_N$, we get $\max_{\mu\in\Xitrain} \Delta_N(\mu) \le \varepsilon_\mathrm{tol}$. In addition, we determine an orthonomal set $\cX_N$ of functions and define the RB trial space as $X_N:=\mathrm{span}(\cX_N)$.

\begin{algorithm}\caption{Greedy Algorithm}\label{Alg:Greedy}
	\begin{algorithmic}[1]
		\STATE choose $\mu_1 \in \cD$ arbitrarily, $\xi_1 \leftarrow u^{\ast,\cN}(\mu_1)$, $N \leftarrow 1$, $S_N:=\{ \mu_1\}$,
			$\cX_N:=\{\xi_1\}$
		\WHILE{$\max_{\mu\in\Xitrain} \tau_N(\mu) \geqslant 1$}\label{alg:2.2}
			\STATE $\mu_{N+1} \leftarrow \mathop{\arg \max}_{\mu\in\Xitrain}  \tau_N(\mu)$
			\STATE $S_{N+1}\leftarrow S_N\cup\{ \mu_{N+1}\}$, $\tilde\xi_{N+1}\leftarrow u^{\ast,\cN}(\mu_{N+1})$
			\STATE orthonormalize $\tilde\xi_{N+1}$ w.r.t.\ $\cX_N$ $\rightarrow \xi_{N+1}$,  
				 $\cX_{N+1} \leftarrow \cX_N\cup \{ \xi_{N+1}\}$
			\STATE $N \leftarrow N + 1$
		\ENDWHILE\label{alg:2.3}
		\STATE $M\leftarrow N$
		\WHILE{$\max_{\mu\in\Xitrain} \Delta_N(\mu) > \varepsilon_\mathrm{tol}$}\label{alg:2.4}
			\STATE $\mu_{N+1} \leftarrow \mathop{\arg \max}_{\mu\in\Xitrain} \Delta_N(\mu)$
			\STATE $S_{N+1}\leftarrow S_N\cup\{ \mu_{N+1}\}$, $\tilde\xi_{N+1}\leftarrow u^{\ast,\cN}(\mu_{N+1})$
			\STATE orthonormalize $\tilde\xi_{N+1}$ w.r.t.\ $\cX_N$ $\rightarrow \xi_{N+1}$,  
				 $\cX_{N+1} \leftarrow \cX_N\cup \{ \xi_{N+1}\}$
			\STATE $N \leftarrow N + 1$
		\ENDWHILE\label{alg:2.5}
	\end{algorithmic}	
\end{algorithm}

\section{Numerical Experiments}\label{Sec:5}
We now present results of some numerical experiments. To this end, we use the following data: $T=1$, $\Omega:=(-150,150)\subset\R$, $d=1$, $\cD=[0,100]$.
In view of Remark \ref{Rem:quad}, we consider quadratic abatement costs $\gamma(t,x)^2 / 2$ which are discounted to the end $T$ of the trading period with an interest rate of $0.05$. This results in the following running costs:
\[
	(f^\gamma(\mu))(t,x) := \frac{\gamma(t,x)^2}2\, e^{0.05(t-T)},
\]

As a (single) parameter $\mu\in\R$, we chose the overall amount of emissions by the involved companies that would arise without any reduction motivated by the EU-ETS, i.e., the ``business as usual emissions''. The control $\gamma(t,x)$ denotes the amount of avoided emissions so that $\mu - \gamma(t,x)$ is the remaining emission, i.e., the drift. 
Together with a constant diffusion term, this yields:
\[
	(A^\gamma(\mu;u)) := -\frac12 u''(t,x) - (\gamma(t,x)-\mu) u'(t,x)
\]
The permits expire worthless at the end of the trading period and the penalty payment is normalized to one currency unit, which is reflected by the terminal condition  $u_T(x):=x^+$ for $x\in\bar\Omega$ in \eqref{Eq:HJB:3}.
The Neuman conditions are modeled by $g_{\textrm{left}}(t)\equiv g_{\textrm{left}}=0$ at $x=-150$ and $g_{\textrm{right}}(t)\equiv g_{\textrm{right}}=1$ at $x=150$.  The function $\gamma\mapsto L^\gamma(\mu;u) := \partial_tu + A^\gamma(\mu; u)$ is $C^1$, so that the optimal control can in fact be characterized as a critical point.

\subsection{Discretization}
The discretization is done by finite differences, fully implicit in time with step size $\Delta t=\frac1{109}$ and using central differences on a regular grid with mesh size $\Delta x=1.5 = \frac{|\Omega|}{200}$. 
In such a discretization, the discrete representation of the derivatives is a linear combination of the value function $u$ at different discretization points. 

The obvious fact that the above defined $A^\gamma(\mu;u)$ and $f^\gamma(\mu)$ are smooth in $\gamma$ and $u$ thus also holds for their discretizations. Due to this smoothness, the assumptions of Lemma \ref{Lemma:Lipschitz} hold.
In order to stabilize  the discrete equations, we use artificial diffusion which is chosen parameter-independent in order to obtain a $\mu$-independent discretization. Clearly, $A^\gamma(\mu;u)$ and also $f^\gamma(\mu)$ are separable in the parameter which also translates to their discrete versions.

\subsection{Estimation of the inf-sup constant}
We start by reporting results concerning the computation of the lower bound $\beta^{\text{LB}}_N(\mu)$ in \eqref{eq:betaLB} for the inf-sup constant using the anchor point strategy described in Algorithm \ref{Alg:infsup} above. We compute an `exact' value of $\beta_N(\mu)$ by solving the corresponding high-dimensional (`truth') generalized eigenvalue problem. The results are shown in Figure \ref{fig_BRR_beta_SCM_testB__10basen_all}. The thin lines represent the intermediate lower bounds which are generated by successive iterations of Algorithm \ref{Alg:infsup}. Different colors represent different stages of the algorithm. The final lower bound $\beta^{\text{LB}}_N(\mu)$ on the training set is marked with the thick line.
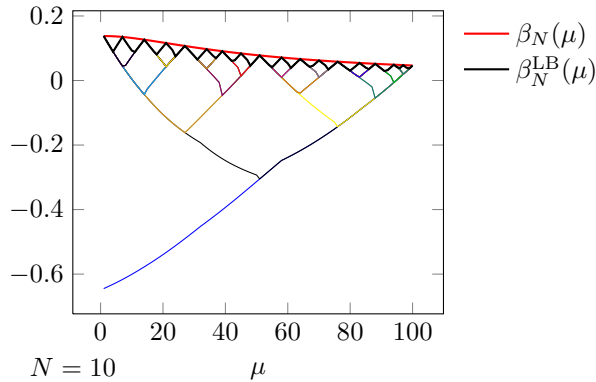
\begin{figure}[htb]
	\begin{center}
		\begin{tikzpicture}
		\begin{axis}[
			width=0.5\textwidth,
			xlabel={$\mu$},
			cycle list name=color list,
			legend pos=outer north east,
			legend cell align=left,
			legend style={draw=none},
			clip=false,
			]
			\addplot+[thick] table {BRR_beta_SCM_testB__10basen_all.txt};
			\foreach \x in {2,3,...,18} {\addplot+[thin] table[y index=\x] {BRR_beta_SCM_testB__10basen_all.txt};}%
			\addplot+[thick, black] table[y index=19] {BRR_beta_SCM_testB__10basen_all.txt};
			\legend{$\beta_N(\mu)$,,,,,,,,,,,,,,,,,,$\beta^{\text{LB}}_N(\mu)$};
			\node[anchor=north] at (xticklabel cs:0){$N=10$};
		\end{axis}
		\end{tikzpicture}
		\caption{Lower bound $\beta^{\text{LB}}_N(\mu)$ for the inf-sup constant $\beta_N(\mu)$. Different colors indicate iterations of the algorithm.}
		\label{fig_BRR_beta_SCM_testB__10basen_all}
	\end{center}
\end{figure}

\subsection{Error estimator}
Next, we test the performance of the error bound $\Delta_N(\mu)$ and the indicator $\tau_N^{\textrm{UB}}(\mu)$ given by the BRR theory. 

To obtain the Lipschitz constant $\varrho$, we use the decomposition given by Lemma \ref{Lemma:Lipschitz}. Since $f$ depends quadratically on $\gamma$ and $L$ depends linearly on $\gamma$ and on $v$ for the EU-ETS, it holds that $\varrho^L_2 = \varrho^f_2 = 0$. Furthermore,
\begin{align*}
	\varrho^f_1 &= \sup_{\gamma^\cN\in\myGammaN} \frac{\| \partial_\gamma f^{\cN,\gamma^\cN} \|_{L(\myGamma,\R)^\cN} }{\| \gamma^\cN \|_{\myGammaN} }
		= \sup_{\gamma^\cN\in\myGammaN} \frac{\| \gamma^\cN e^{0.05(t-T)} \|_{L(\myGamma,\R)^\cN} }{\| \gamma^\cN \|_{\myGammaN} }, \\
	\varrho^L_0 &= \sup_{\gamma^\cN\in\myGammaN} \frac{\| L^{\cN,\gamma^\cN} \|_{L(\R^\cN,\R^\cN)} }{\| \gamma^\cN \|_{\myGammaN} }, \text{ and}\quad
	\varrho^L_1 = \sup_{x^\cN\in\X^\cN} \frac{\| \partial_\gamma L^\cN u^\cN \|_{L(\myGamma,\R)^\cN} }{\| x^\cN \|_{\X^\cN} },
\end{align*}
which can be computed solving generalized eigenvalue problems. Since these quantities are independent of $\mu$, is suffices to determine them once in the offline phase.

The norm of the residual $\| \cG^\cN(\mu)(x_N^\ast(\mu))\|_{\Y^\cN}$ can be computed efficiently in the online phase using the separation in the parameter.
Inserting \eqref{Eq:paramsep} and $x_N(\mu) = \sum_{n=1}^N \bx_N^n(\mu) \, \xi_n$ into $\langle \cG^\cN(\mu)(x_N^\ast(\mu)), \cG^\cN(\mu)(x_N^\ast(\mu))\rangle_{\Y^\cN}$ results in a sum of products of $\mu$-independent functions which can be evaluated fast in the online phase and $\Y^\cN$-dot-products which are precomputed in the offline phase.

In Figure \ref{fig_BRR_beta_SCM_testB__konvergenz}, we compare the true error (again computed with respect to the detailed discretization) with the indicator, the error bound and the size of the residual. The values correspond to the maximum of the corresponding quantities over a test sample in $\cD$. We recall that the BRR-based error estimator is only meaningful for $\tau \le 1$, which is here the case for $N\ge 10$.
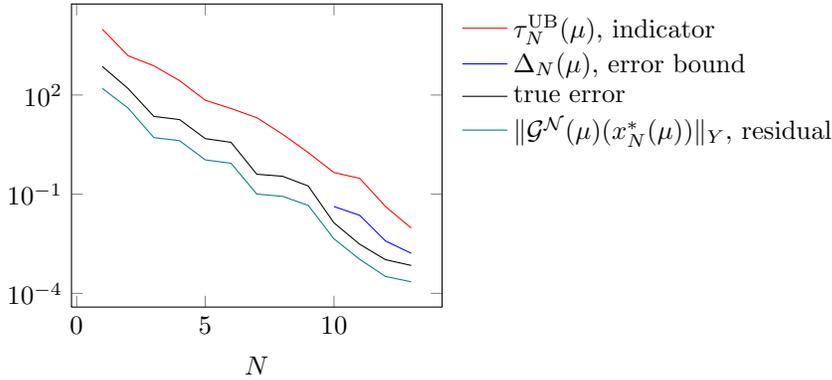
\begin{figure}[htb]
	\begin{center}
		\begin{tikzpicture}
		\begin{axis}[
			width=0.5\textwidth,
			xlabel={$N$},
			ymode=log,
			cycle list name=color list,
			unbounded coords=jump,
			legend pos=outer north east,
			legend cell align=left,
			legend style={draw=none},
			]
			\addplot table {BRR_beta_SCM_testB__konvergenz.txt};
			\addlegendentry{$\tau_N^{\textrm{UB}}(\mu)$, indicator};
			\addplot table[y=err_est] {BRR_beta_SCM_testB__konvergenz.txt};
			\addlegendentry{$\Delta_N(\mu)$, error bound};
			\addplot table[y=err] {BRR_beta_SCM_testB__konvergenz.txt};
			\addlegendentry{true error};
			\addplot[teal] table[y=res] {BRR_beta_SCM_testB__konvergenz.txt};
			\addlegendentry{$\|\cG^\cN(\mu)(x_N^\ast(\mu))\|_Y$, residual};
		\end{axis}
		\end{tikzpicture}
		\caption{Error, indicator, bound and residual over $N$.}
		\label{fig_BRR_beta_SCM_testB__konvergenz}	
	\end{center}
\end{figure}

In order to investigate the effectivity of the error bound, we fix $N=10$ and vary the parameter $\mu$ over the parameter range $\cD=[0,100]$. We can see in Figure \ref{fig_BRR_beta_SCM_testB__10_err_est} that indicator, bound and residual are numerically zero for the snapshots, where the true error of course also vanishes. The effectivity of the error bound, i.e., the ratio of error estimator and true error, is shown in Figure \ref{fig_BRR_beta_SCM_testB__10_effektivitaet}. We expect a growth with respect to increasing $\mu$ since the PDE starts becoming increasingly convection-dominated. However, the maximum size is below $8$ which seems a reasonable size to us. Of course, well-known techniques for stabilization as well as parameter-adaptivity may additionally be used, e.g.\ \cite{Welper,Eftang}.

\begin{figure}
	 \begin{subfigure}[t]{0.48\textwidth}
		\begin{tikzpicture}
		\begin{axis}[
			width=0.95\textwidth,
			xlabel={$\mu$},
			ymode=log,
			cycle list name=color list,
			unbounded coords=jump,
			legend pos=outer north east,
			legend cell align=left,
			legend style={draw=none,font=\tiny},
			clip=false,
			]
			\addplot table {BRR_beta_SCM_testB__10_err_est.txt};
			\addlegendentry{$\tau_N^{\textrm{UB}}(\mu)$};
			\addplot table[y=err_est] {BRR_beta_SCM_testB__10_err_est.txt};
			\addlegendentry{$\Delta_N(\mu)$};
			\addplot table[y=err] {BRR_beta_SCM_testB__10_err_est.txt};
			\addlegendentry{error};
			\addplot[teal] table[y=res] {BRR_beta_SCM_testB__10_err_est.txt};
			\addlegendentry{residual};
		\end{axis}
		\end{tikzpicture}
		\caption{Error, indicator, bound and residual over $\mu\in\cD$.}
			\label{fig_BRR_beta_SCM_testB__10_err_est}
	\end{subfigure}
	\hspace*{9mm}
	 \begin{subfigure}[t]{0.48\textwidth}
	 	\begin{tikzpicture}
		\begin{axis}[
			width=0.95\textwidth,
			xlabel={$\mu$},
			unbounded coords=jump,
			legend pos=outer north east,
			legend cell align=left,
			legend style={draw=none},
			clip=false,
			]
			\addplot[blue] table {BRR_beta_SCM_testB__10_effektivitaet.txt};
		\end{axis}
		\end{tikzpicture}
		 \caption{Effectivity of the error bound.}
		 \label{fig_BRR_beta_SCM_testB__10_effektivitaet}
	\end{subfigure}
	\caption{Effectivities over parameter range $\cD=[0,100]$ for $N=10$.}
\end{figure}
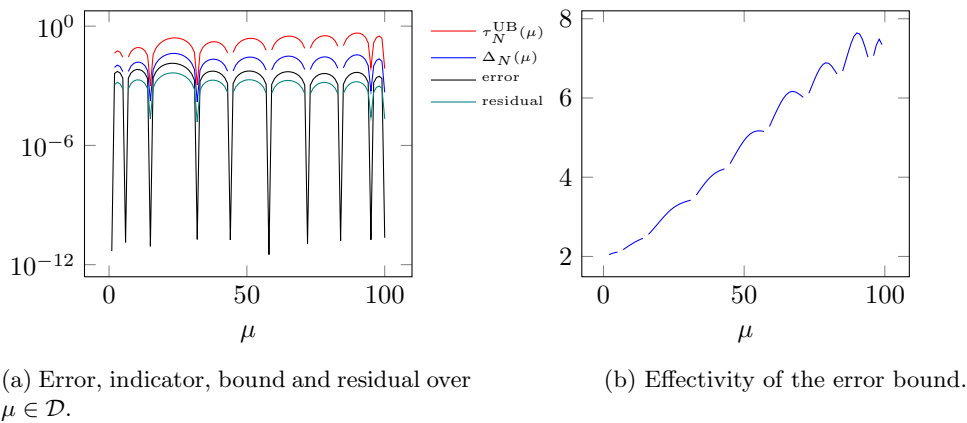

\section{Summary}

We have presented a Reduced Basis Method (RBM) for rapidly solving the parameterized Hamilton-Jacobi-Bellman (HJB) equation with the specific application to the European Union Emission Trading Scheme (EU-ETS). In particular, we have introduced a rigorous bound of the error with respect to an online-efficient error estimator. The involved parameter-dependent constants can be computed online-efficient by an anchor-point based Successive Constraint Method. Numerical experiments confirm the effectivity of the estimator.

Future research will be devoted to the specific application of the HJB-RBM to the EU-ETS in order to determine optimal regulatory strategies. From the mathematical point of view, we will consider extensions to more general settings, i.e., more general nonlinearities and stronger convection in the operator.


\bibliography{bibliography}
\bibliographystyle{siam}

\end{document}